\documentclass[11pt,twoside]{amsart}
\usepackage{bbm}
\usepackage{verbatim}
\usepackage{graphicx,color}
\usepackage{enumerate}
\usepackage[all]{xy}
\usepackage[hyperindex=true,plainpages=false,colorlinks=false,pdfpagelabels]{hyperref}

\theoremstyle{plain}
\newtheorem{The}{Theorem}
\newtheorem*{The*}{Theorem}

\newtheorem{Lem}{Lemma}

\newtheorem*{Cor*}{Corollary}

\theoremstyle{definition}

\newtheorem{Rem}{Remark}
\newtheorem{Exa}{Example}
\newtheorem*{Rem*}{Remark}

\numberwithin{equation}{section}

\DeclareMathOperator{\End}{End}

\DeclareMathOperator{\SL}{SL}

\DeclareMathOperator{\SU}{SU}
\DeclareMathOperator{\Span}{Span}

\newcommand{\dvector}[1]{{\left(\begin{matrix}#1\end{matrix}\right)}}
\DeclareMathOperator{\dbar}{\bar\partial}

\newcommand{\R}{\mathbb{R}}
\newcommand{\Q}{\mathbb{Q}}
\newcommand{\C}{\mathbb{C}}
\newcommand{\N}{\mathbb{N}}
\newcommand{\Z}{\mathbb{Z}}

\newcommand{\CP}{\mathbb{CP}}

\begin{document}

\title{The spectral curve theory for $(k,l)$-symmetric CMC surfaces}

\author{Lynn Heller}

 \author{Sebastian Heller}

\author{Nicholas Schmitt}

\address{Lynn Heller \\
  Institut f\"ur Mathematik\\
 Universit{\"a}t T\"ubingen\\ Auf der Morgenstelle
10\\ 72076 T¬ubingen\\ Germany
 }
 \email{lynn-jing.heller@uni-tuebingen.de}

\address{Sebastian Heller \\
  Institut f\"ur Mathematik\\
 Universit{\"a}t T\"ubingen\\ Auf der Morgenstelle
10\\ 72076 T¬ubingen\\ Germany
 }
 \email{heller@mathematik.uni-tuebingen.de}
 
\address{Nicholas Schmitt \\
  Institut f\"ur Mathematik\\
 Universit{\"a}t T\"ubingen\\ Auf der Morgenstelle
10\\ 72076 T¬ubingen\\ Germany
 }
\email{nschmitt@mathematik.uni-tuebingen.de}




\thanks{The first author is supported by the European Social Fund, by the Ministry of Science, Research and the Arts Baden-W\"urtemberg and by the Baden-W\"urtemberg Foundation, the other authors are supported by the DFG through the project HE 6829/1-1.}

\begin{abstract}
Constant mean curvature surfaces in $S^3$ can be studied via their associated family of flat connections. In the case of tori this approach has led to a deep understanding of the moduli space of all CMC tori.
For compact CMC surfaces of higher genus the theory is far more involved due to the non abelian nature of their fundamental group. In this paper we extend the spectral curve theory for tori developed in \cite{Hi, PiSt} and for genus $2$ surfaces \cite{He3} to CMC surfaces in $S^3$ of genus $g=k\cdot l$ with commuting $\Z_{k+1}$ and $\Z_{l+1}$ symmetries. We determine their associated family of flat connections via certain flat line bundle connections
parametrized by the spectral curve.
We generalize the flow on spectral data introduced in \cite{HeHeSch} 
and prove the short time existence of this flow for certain families of initial surfaces. In this way we obtain various families of new
CMC surfaces of higher genus with prescribed branch points and prescribed umbilics.
 \end{abstract}
  
 \maketitle


\section{Introduction}
\label{sec:intro}

Surfaces of constant mean curvature are stationary points of the area functional under a volume constraint. Dropping this constraint yield minimal surfaces - those of zero mean curvature -  which are in  Euclidean space never compact as a consequence of the harmonicity of the conformal immersion. In this paper we are interested in compact CMC surfaces in the round $3$-sphere. These surfaces are related to CMC surfaces in other space forms by the Lawson correspondence \cite{L}.

In analogy to euclidean minimal surfaces, CMC immersions  from simply connected domains into space forms can be parametrized via a generalized Weierstrass representation \cite{DPW}.  But global questions concerning the construction of closed surfaces of genus $g\geq2$ are very hard to tackle in this setup as the space of conformal
 CMC immersions of a disc $D$ or a plane is always infinite dimensional. Searching for periodic CMC immersions
 which factor through a compact quotient $D\to D/\Gamma=M$ of genus $g\geq2$ is worse than searching for a needle in a haystack, in particular because
 the Fuchsian group $\Gamma$ is non-abelian.
 In contrast, the abelian torus case  is comparably well understood by applying integrable systems methods introduced in the work of  Abresch \cite{Ab}, Pinkall and Sterling \cite{PiSt}, Hitchin \cite{Hi} and Bobenko \cite{B}  in the 1980s.
These methods were used to produce various new examples of CMC tori.
Recently Hauswirth, Kilian and Schmidt studied the moduli space of all
minimal tori in $S^2\times \R$ by integrable systems methods. They proved that properly embedded minimal annuli
in $S^2\times\R$ are foliated by circles \cite{HaKiSch, HaKiSch1}.
 The technique (based on the so-called Whitham deformation of the spectral data) 
 has also been applied for the investigation
of Alexandrov embedded CMC tori in $S^3$ \cite{HaKiSch2}, giving an alternative approach to the conjectures
of Lawson (proved by Brendle \cite{B}) and  Pinkall-Sterling (proved by Andrews and Li \cite{AL} using Brendle's method).

The situation is completely different for compact  higher genus CMC surfaces, where only few examples are known. Lawson \cite{L} constructed closed embedded minimal surfaces in the round 3-sphere for every genus and
Kapouleas \cite{Ka1,Ka2} showed the existence of compact CMC  surfaces in Euclidean 3-space for all genera.
All known examples have been constructed implicitly, and even fundamental geometric properties like the area
cannot be explicitly computed.

In the integrable system approach to CMC surfaces  a
$\C^*$-family of flat 
$\SL(2,{\C})$-connections $\lambda\in \C^*\mapsto\nabla^\lambda$ is associated to the immersion. 
Knowing the family of flat connections is equivalent to knowing the CMC surface itself, as the surface is given by the gauge
between two trivial connections $\nabla^{\lambda_1}$ and $\nabla^{\lambda_2}$
 for 
$\lambda_1\neq\lambda_2\in S^1\subset\C^*$  with mean curvature $H=i\frac{\lambda_1+\lambda_2}{\lambda_1-\lambda_2}.$ By construction this family of flat connections is unitary for all $\lambda \in S^1$. 
For tori, due to their abelian fundamental group, $\nabla^\lambda$ splits for generic $\lambda\in\C^*$ into a direct sum of flat line bundle connections. 
Thus for a CMC torus the associated $\C^*$-family of flat $\SL(2,\C)$-connections 
is given by spectral data parametrizing the corresponding family of flat line bundles on the torus. If the closed CMC surface has higher genus the associated flat $\SL(2,\C)$-connections are generically irreducible and therefore 
the abelian spectral curve theory for CMC tori is no longer applicable.
Nevertheless,  for Lawson symmetric  CMC surfaces of genus $2$ \cite{He3} it is still possible to  characterize flat symmetric 
$\SL(2,\C)$-connections in terms of flat line bundle connections
via abelianization  of the corresponding flat connections. 
Hence, the associated family $\nabla^{\lambda}$ 
of a symmetric CMC surface of higher genus is again
determined by spectral data which parametrize flat line bundle connections on a torus. But the unitarity condition for $\nabla ^{\lambda}$ along the unit circle $\lambda \in S^1$ is only given implicitly in terms of the Narasimhan-Seshadri section. 
In order to find a way of  constructing spectral data satisfying the unitarity condition, a flow on the spectral data was introduced in \cite{HeHeSch} restricting to surfaces corresponding to Fuchsian systems on a $4$ punctured sphere with the same pole behavior at each singularity. The flow is designed to deform a known initial CMC surface in a direction which "continuously" change the genus of the surface but preserves the constant mean curvature property
as well as certain intrinsic and extrinsic closing conditions. For rational times the deformed surface closes to a compact (but possibly branched) CMC surface and reaches an immersed CMC surface at integer times. As a corollary of the short time existence of the flow for certain initial values new families of compact branched CMC surfaces in $S^3$ are obtained. The purpose of this paper is to generalize the results in \cite{He3} and \cite{HeHeSch} to a more general class of surfaces, namely those surfaces of genus $g$ with commuting $\Z_{k+1}$ and  $\Z_{l+1}$ symmetries of genus
 $g = k\cdot l.$ Examples of surfaces with this properties are Lawson's minimal surfaces $\xi_{k,l},$ see \cite{L}. Figure \ref{fig:lawson} illustrates the flow from the Lawson surface of genus $2$, $\xi_{2,1},$ to the Lawson surface of genus $4$, $\xi_{2,2}.$  It differs slightly from the proposed flow which 
 goes from the Clifford torus directly to the Lawson surface $\xi_{2,2}.$ Instead of flowing into the $(k, l)$-direction, we go first from the Clifford torus $\xi_{1,1}$ to the Lawson surface $\xi_{2,1}$ and from there to $\xi_{2,2}.$  But it should be mentioned that the flows into the
 $k$ and into the $l$ direction are locally commuting where they exist (see also Remark \ref{3d}), and therefore the picture is still adequate to illustrate our theory.

The paper is organized as follows: We first describe the integrable systems setup for CMC surfaces in which the non linear PDE $H \equiv const$ is translated into  a system of linear ODEs, namely a $\C^*$-family of flat $SL(2, \C)$-connections and how to reconstruct a CMC surface from a given family of connections. Then we describe the Riemann surface structures of $(k,l)$-symmetric surfaces. In Section \ref{Abelianization} we summarize the abelianization procedure for Fuchsian systems on a $4$-punctured sphere carried out in \cite{HeHe} and show how to relate a $\C^*$-family of Fuchsian systems 
satisfying certain properties to the associated family of flat connections of a CMC surface with boundary. Under some rationality condition the analytic continuation of the CMC surface close to a compact (and possibly branched) CMC surface with $(k,l)$-symmetry (Section  \ref{Rational_weights}). 
In Section  \ref{whitham flow} we first describe the spectral data for the CMC tori which are the initial points of our flow.
Then we define the flow on these spectral data for every direction $q\in \R$. The short time existence of the flow  follows immediately from \cite{HeHeSch}. Thus for initial data given by homogeneous CMC tori we obtain a unique flow for every $q \in \R$. In particular, we have that flowing the Clifford torus in direction $q\in \Q$ is equivalent to the flow of Plateau solutions by changing 
the angle of the fundamental piece in Lawson's construction. The flow reaches the minimal surfaces $\xi_{k,l}$ for $q = \tfrac{(2k+2)(l-1)}{(2l+2)(k-1)}.$
Further, for every direction $q\in\mathbb R,$ there are two distinct flows starting at the  
2-lobed Delaunay tori of spectral genus $1$. For every initial surface, each direction $q\in \Q$ and every rational time, we obtain thus a closed CMC surface with prescribed branch points and umbilics.

\section{CMC Surfaces via Integrable Systems Methods}
We consider conformal immersions $f\colon M \to S^3$ from a compact Riemann surface to the round $3$-sphere with constant mean curvature $H$. In integrable surface theory the elliptic PDE $H \equiv const.$ is studied via a system of linear ODEs on the Riemann surface, namely a $\C^*$-family $\nabla^{\lambda}$ of flat $\SL(2, \C)$ connections on the trivial $\C^2$ bundle over $M$, see \cite{Hi} or \cite{B, He1}. 

\begin{The}[\cite{Hi,B}]
\label{The1}
Let $f\colon M\to S^3$ be a conformal CMC immersion. Then there exists
an associated family of flat $\SL(2,\C)$-connections
\begin{equation}\label{associated_family}
\lambda\in\C^*\mapsto \nabla^\lambda=\nabla+\lambda^{-1}\Phi-\lambda\Phi^*
\end{equation}
on a hermitian rank $2$ bundle $V\to M,$ where $\Phi$ is a nilpotent and nowhere vanishing
complex linear $1$-form and $\Phi^*$ is its
adjoint. The family $\nabla^\lambda$
is unitary along $ S^1\subset\C^*$ and trivial for
$\lambda_1\neq\lambda_2\in  S^1$.  \\

Conversely, given such a family of flat connections,  the gauge between the trivial $SU(2)$ connections
$\nabla^{\lambda_1}$ and $\nabla^{\lambda_2}$ is an immersion into $S^3=\SU(2)$ of constant mean curvature $H=i\frac{\lambda_1+\lambda_2}{\lambda_1-\lambda_2}.$ \end{The}
\begin{Rem}
The unitarity condition for $\nabla^{\lambda}$ for $\lambda \in S^1$ is necessary to obtain a CMC map into $S^3=SU(2),$ and the
fact that these connections are well-defined on the surface is
equivalent to the intrinsic closing condition, since it yields a well-defined metric on $M$ which gives rise to a solution of the Gauss-Codazzi equations. The Sym point condition or the extrinsic closing condition is that there exist $\lambda_1, \lambda_2 \in S^1$ such that the connections $\nabla^{\lambda_i}$ have trivial monodromy. The surface closes on a covering of $M$ if the monodromy at the Sym points is a rotation by a rational angle.
\end{Rem}

\begin{The}[\cite{Hi, He1, Ge}]
The $\C^*$-family of flat $SL(2, \C)$-connections $\nabla^{\lambda}$ is trivial for all $\lambda$ if and if the surface is totally umbilic. Further, the family of connections $\nabla^{\lambda}$ is reducible for all $\lambda$ if and only if the surface is a (branched covering of a) CMC torus.
\end{The}

The theorem explains why integrable systems methods have only been very successful in the study of CMC tori. For $M$ being a torus
 $\nabla^{\lambda}$ splits generically into the direct sum of two line bundle connections over $M$. To be more concrete: For generic $\lambda_0 \in \C_*$ there is a neighborhood of $\lambda_0$ and a splitting $\underline \C^2 = L_\lambda^+ \oplus L_\lambda^-$ such that 
 $\nabla^{\lambda}$ is gauge equivalent to
\[ d + \begin{pmatrix} -\chi(\lambda) d\bar w + \alpha(\lambda) dw & 0 \\ 0 & \chi(\lambda) d\bar w - \alpha(\lambda) dw \end{pmatrix},\]
where $dw$ is the holomorphic $1$-form on the torus $M$ and $\chi$ and $\alpha$ are holomorphic functions in $\lambda$ parametrizing the induced holomorphic and anti-holomorphic structures of $L_\lambda.$ It is shown in \cite{Hi} that the splitting fails only at finitely many points $\lambda_i \in \C_*$, where the eigenlines of the monodromy $L_\lambda^\pm$ coalesce. This can only happen if the connections on $L_{\lambda_i}^\pm$ are self-dual, i.e., if the holomorphic line bundle $L_{\lambda_i}$ over $M$ is a spin bundle. By replacing the parameter plane $\C^*$ by a double covering $\Sigma$ of $\CP^1$ branched at the exceptional $\lambda_i$ and at $\lambda= 0, \lambda= \infty,$ globally defined holomorphic maps $\chi$ and $\alpha$ can be obtained \cite{Hi}. The (compact) hyper-elliptic Riemann surface $\Sigma$ is called the spectral curve of the CMC torus. Note that by demanding $\nabla^{\lambda}$ to be unitary on the unit circle, 
the spectral curve $\Sigma$ inherits a real involution covering $\lambda\mapsto\bar\lambda^{-1}$ and
the holomorphic map $\alpha$ is already determined by the map $\chi.$ As was shown in \cite{Hi} the family of 
flat $SL(2, \C)$-connections $\nabla^{\lambda}$ and hence the corresponding CMC torus can be reconstructed from the holomorphic data
$(\chi,E)$ on $\Sigma,$ where $E$ is a point in the Picard variety of $\Sigma$ determined by evaluating the parallel eigenlines
$L_\xi, \xi\in\Sigma$
at a fixed $p\in T^2.$

If $M$ is of higher genus $g\geq 2$, then we cannot obtain interesting examples by looking at families of reducible connections. Thus there is no straight forward generalization of the theory developed in \cite{Hi, PiSt} to higher genus surfaces. Nevertheless, integrable systems methods can be applied \cite{He1,He2, He3, HeHeSch} to obtain a better understanding of higher genus CMC surfaces.
 It has been proven very useful to first determine only the gauge equivalence classes $[\nabla^{\lambda}]$ of the connections. Thus let 
$\mathcal A^2=\mathcal A^2(M)$  be the moduli space of flat $\SL(2,\C)$-connections modulo gauge transformations.  The space $\mathcal A^2$ inherits the structure of a complex analytic variety of dimension $6g-6$ whose singularity set consists of the gauge classes of reducible connections \cite{G, Hi1}. Note that a necessary condition for the Sym point condition to hold is that the corresponding $\nabla^{\lambda_i}$ are reducible.

For a CMC surface $f$ with associated family of flat connections $\nabla^{\lambda}$ consider the map
\[\mathcal D\colon \C^* \to\mathcal A^2,\; \lambda\mapsto[\nabla^\lambda].\]
The map $\mathcal D$ does not uniquely determine a CMC surface in $S^3$ in general but those CMC surfaces
corresponding to the same $\mathcal D$ are related by a well understood transformation called dressing, see \cite{He3}. In the case where $\nabla^{\lambda}$ is generically irreducible, there are only finitely many dressing transformations (Theorem 7 in  \cite{He3}), since they correspond to reducible connections in the family $\nabla^{\lambda}$. For CMC tori, these dressing transformations are
generically the isospectral deformations induced by a shift of the eigenline bundle, see \cite{McI}. 

\begin{The}\label{lifting_theorem}
Let $M$ be a Riemann surface and let $\mathcal D\colon\C^*\to\mathcal A^2(M)$ be a holomorphic map such that

\begin{enumerate}
\item the unit circle $S^1\subset\C^*$ is mapped
into the real analytic subvariety consisting of gauge equivalence classes of unitary flat connections,
\item around $\lambda=0$ there exists a  local lift $\tilde\nabla^\lambda$ of $\mathcal D$ with an expansion $$\tilde\nabla^\lambda\sim\lambda^{-1}\Psi+\tilde\nabla^0 +  \text{Êhigher order terms in } \lambda$$ for a nilpotent $\Psi\in\Gamma(M,K\End_0(V)),$
\item there are two distinct points
$\lambda_{1,2}\in S^1\subset\C^*$ such that $\mathcal D(\lambda_k)$ represents the trivial gauge class,
\end{enumerate}
then
there exists a (possibly branched) CMC surface $f\colon M\to S^3$ inducing the map $\mathcal D$. Further, 
the surface $f$ is unique up to dressing transformations and
the branch points of $f$ are the zeros of $\Psi.$
\end{The}
This theorem was proven for generically irreducible connections in \cite{He3} and in the case of reducible connections it follows from \cite{Hi}.
\begin{Rem}
 Since $\mathcal D$ maps into the real analytic subvariety consisting of gauge equivalence classes of unitary flat connections
 along the unit circle,  $\mathcal D$ is already determined by its induced family of gauge classes of holomorphic structures 
 $(\nabla^\lambda)''$ on the closed unit disc centered at $\lambda= 0$
by the Schwarzian reflection principle. \end{Rem}

\section{$(k,l)$-symmetric Riemann surfaces}\label{kl-RS}

  We consider compact CMC surfaces of genus $g=k\cdot l$ equipped with
commuting $ \Z_{k+1}$ and $ \Z_{l+1}$ symmetries preserving the orientation of the surfaces and the ambient space.
 Moreover, we assume that $\Z_{k+1}$ has $2l+2$ fix points, and $\Z_{l+1}$ has $2k+2$ fix points, all of total branch order $k$ and $l$ respectively. Examples
 of $(k,l)$-symmetric CMC surfaces are provided by Lawson \cite{L} - the Lawson minimal surfaces $\xi_{k,l}.$
 
 We want to compute the Riemann surface structure $M$ as well as the Hopf differential of  $(k,l)$-symmetric CMC surfaces.
Note that the Riemann surfaces $M/ \Z_{k+1}$ and $M/ \Z_{l+1}$ are both the Riemann sphere
 as a consequence of the Riemann-Hurwitz formula. Let us denote
 the fixed points of the $\Z_{k+1}$-action by $Q_1,..,Q_{2l+2}$ and the fixed points of $\Z_{l+1}$ by $P_1,..,P_{2k+2}.$ Then
 $\CP^1=M/ \Z_{k+1}$ has $2l+4$ marked points,
 the images of $Q_1,..,Q_{2l+2}$ which we denote by the same symbols again, and the images $p_+,p_-$ of the set $\{P_1 , .. , P_{2k+2}\}.$
 That we have exactly two points $p_+$ and $p_-$ follows from the fact that $\Z_{k+1}$ and $\Z_{l+1}$
 are commuting. Another consequence of this assumption is that $\Z_{l+1}$ acts on $\CP^1=M/ \Z_{k+1}$ with fix points $p_+,p_-,$ and such that the points
 $Q_1,..,Q_{2l+2}$ lie in exactly two orbits of the $\Z_{l+1}$-action. 
 
 Conversely, to obtain a $(k,l)$-symmetric Riemann surface we start with the Riemann sphere with 4 marked points,  which we can assume without loss of generality
 to be $p_+=0,$ $p_-=\infty,$ and $q_+=1$, $q_-=m\in\C.$
 Then we take the $(l+1)$-fold covering\[\CP^1\to\CP^1,\,\,\, z\mapsto z^{l+1},\]
 such that the preimages of $q_+,q_-$ are $2l+2$ points which we denote by $Q_1,..,Q_{2l+2}\in\CP^1(=M/ \Z_{k+1}).$ We define $M$  to be the $(k+1)$-fold covering of this  $\CP^1$ with $2l+4$ marked points as indicated in the  following diagram:

\begin{xy}\label{diag:MP1P1}
\hspace{5.0 cm}
  \xymatrix{
              M \ar[d]^{\mod \Z_{k+1}}  \\
           M/ \Z_{k+1}=\CP^1 \ar[d]^   {\mod \Z_{l+1}\cong\, z\mapsto z^{l+1}} \\
             (M/ \Z_{k+1})/\Z_{l+1}=\CP^1   . }
\end{xy}

The Hopf differential $Q$ of a $(k,l)$-symmetric CMC surface is invariant under $\Z_{k+1}$ as well as $\Z_{l+1}.$ 
Hence it is the pull-back of a meromorphic quadratic differential
on $\CP^1$ with at most simple poles at $z=0,1,m,\infty$ and no other poles. Since the CMC surface is not totally umbilic,  $Q$ is the pull-back of
 a non-zero multiple of
\[\frac{(dz)^2}{z(z-1)(z-m)}.\]

By similar arguments it is possible to prove that its associated $\C^*$-family of flat 
connections can be pushed down
to a family of singular connections on the $4$-punctured sphere. This has been carried out in detail for the Lawson surface $\xi_{2,1}$ 
in \cite{He2}. Point-wise in $\lambda$ it can also be deduced from \cite{Bis2}.
\begin{Lem}\label{Solving-an-ODE}
Consider the $(k+1)(l+1)$-fold covering $h\colon M\to\CP^1$ branched over 4 points $0,1,m$ and $\infty \in \CP^1.$ Assume as above  that
$h$ branches over $0$ and $\infty$ with order $l,$ and over $1,m$ with order $k.$ 
Then, the pull-back by $h$ of the connection
\[ d\pm \frac{n_0}{l+1}\frac{dz}{z}\pm \frac{n_1}{k+1}(\frac{dz}{z-1}-\frac{dz}{z-m})\]
has trivial monodromy on $M$ for all $n_0, n_1\in\Z.$
\end{Lem}
\begin{proof}
Because the $\C*$-valued monodromy representation of $d\pm \frac{n_0}{l+1}\frac{dz}{z}\pm \frac{n_1}{k+1}(\frac{dz}{z-1}-\frac{dz}{z-m})$ is given by
the multiplication of the monodromy representations of the connections $d\pm \frac{n_0}{l+1}\frac{dz}{z}$ and $d\pm \frac{n_1}{k+1}(\frac{dz}{z-1}-\frac{dz}{z-m}),$ 
it is enough to prove the assertion for either $n_0=0$ or $n_1=0.$ By symmetry, we can simply restrict to the case of $n_1=0.$
Let $\tilde h$ denote the map $M\to\CP^1=M/\Z_{k+1}$ with $p_+=0,\, p_-=\infty$. Then
\[\frac{d\tilde h}{\tilde h}=\frac{1}{l+1}\frac{d h}{h}=h^*(\frac{1}{l+1}\frac{d z}{z})\]
and we obtain that  $\tilde h^{\mp n_0}$ is a globally well-defined parallel section of the connection $d\pm \frac{n_0}{l+1}\frac{dz}{z}$ on $M$.
\end{proof}
\section{Abelianization of Fuchsian Systems and $(k,l)$-symmetric surfaces}\label{Abelianization}  
 In the second section we have described the integrable systems approach to CMC surfaces via their
 associated family of flat $SL(2, \C)$-connections.  Symmetries of the surface induce symmetries of the associated family viewed as a complex curve in $\mathcal A^2(M)$. In this section we want to study CMC surfaces  such that their associated family of $\SL(2, \C)$-connections is gauge equivalent (by a $\lambda$-dependent gauge)
 to a family of Fuchsian Systems on a $4$-punctured sphere. In this case, although the connections $\nabla^{\lambda}$ do not split into line bundle connections, abelianization \cite{He3, HeHe} gives a $2$-to-$1$ correspondence between $\mathcal A^1(T^2)$ the space of flat line bundle connections  on a torus $T^2$ and $\mathcal A^2(M)$. Thus the map $\mathcal D$ needed in Theorem \ref{lifting_theorem} can be determined by families 
 of line bundle connections on $T^2$. However, we need to pay special attention 
 that the $2$-to-$1$ correspondence has branch points, and we are again forced to use a spectral curve $\Sigma$ in order to
 parametrize families of 
 line bundle connections giving rise to the associated $\C^*$-families of CMC surfaces. Moreover, we need to guarantee that $\mathcal D$  maps the preimage of the unit circle (inside the spectral curve) to gauge equivalence classes of unitary flat connections.
 
We need to fix some notations first. A Fuchsian System on the  Riemann sphere with $4$ punctures $z_0 = 0, z_1 = 1, z_2 = m, z_3 = \infty,$ 
for some $m \in \C \setminus \{0,1\}$ 
is a meromorphic connection on the trivial bundle $V= \underline{\C}^2$ of the form
 \begin{equation}\label{Fuchs}\nabla = d + A_0\frac{dz}{z} + A_1\frac{dz}{z-1} + A_2\frac{dz}{z-m},\end{equation}
 
 where we choose $A_i \in \mathfrak{sl}(2, \C).$  In this setup $\nabla$ has automatically a singularity in $z = \infty$ with residue $A_3 = -A_0-A_1-A_2.$
 
The conjugacy class of the local monodromy around a puncture $z_i$ is determined by the eigenvalues $\pm \hat \rho_i$ of $A_i.$ Necessary conditions for the monodromy representation to be unitary up to conjugation (unitarizable monodromy) are posed by \cite{Bis3}, the so called Biswas conditions. We restrict to the case where $\hat\rho_i \in ]0, \tfrac{1}{2}[$ 
here, and for explaining the Biswas condition 
the notion of a parabolic structure turns out to be useful.

A Fuchsian system \eqref{Fuchs} on a punctured sphere naturally defines a parabolic structure as follows: Let $V: \C^2 \rightarrow \CP^1$ be the trivial holomorphic bundle and let $E_i$ denote the eigenspace of the residue $A_i \in \mathfrak{sl}(2, \C)$ of $\nabla$ at $z_i$ with respect to the positive eigenvalue $\hat\rho_i$. A parabolic structure on $V$ is given by a filtration of the fibers of $V$ at the singular points $z_i$
 \[0 \subset E_i \subset V_{z_i}\]
 together with a weight filtration $(\hat\rho_i, - \hat\rho_i),$ i.e., the line $E_i$ is equipped with the weight $\hat\rho_i$ and $V_{z_i} \setminus E_i$ is equipped with the weight $- \hat\rho_i,$ see for example \cite{MSe, Pi}. To a holomorphic  line subbundle $L\subset V$ we assign its parabolic degree
 \[ \text{par-deg} L = \deg L + \sum_{i=0}^3 \gamma_i,\]
 where either $\gamma_i = \hat\rho_i$ if $L_{z_i} = E_i,$ or $\gamma = -\hat \rho_i$. Equivalent Fuchsian systems (by conjugation) yield equivalent parabolic structures.  A parabolic structure is called (semi-)stable if the parabolic degree of every holomorphic line subbundle $L$ is strictly negative ($\leq 0$). 
 The correspondence between stable parabolic structures  and   irreducible Fuchsian systems with unitarizable monodromy is $1$-to-$1$ by the Riemann-Hilbert correspondence for $SL(2, \C)$ and by the Mehta-Seshadri correspondence \cite{MSe}. 

Two Fuchsian systems inducing the same parabolic structure differ by a parabolic Higgs field, i.e., a meromorphic $1$-form $$\Psi \in H^{1,0}(\CP^1 \setminus\{z_0, ..., z_3\}, \mathfrak{sl}(2, \C))$$
with first order poles at $z_i$ such that $E_i$ lies in the kernel of the residues of $\Psi.$ Hence, the
determinant of $\Psi$ has at most first order poles at the singularities.
\subsection{Abelianization}
In the abelian case of a reducible $SL(2, \C)$ connection, the connection is (generically) 
given as the direct sum of line bundle connections.
If $\nabla$ is irreducible, the Fuchsian system is still determined by certain flat line bundle connections.
The line bundles are the eigenlines of 
 a  parabolic Higgs field $\Psi$ with non-zero determinant. The correspondence between these line bundle connections and $\nabla$ extends to the exceptional Fuchsian systems which do not permit a Higgs field
 with non-zero determinant. We briefly summarize the relevant results of \cite{HeHe} here. 

Let $\nabla$ be a Fuchsian system as defined in \eqref{Fuchs} with (semi-)stable parabolic structure. We first consider the case where $\nabla$ has a parabolic Higgs field $\Psi$ with non-zero determinant. Since the residues of $\Psi$ are nilpotent and because $deg(K^2) = -4$ for $\CP^1$,  there exists a $c \in \C_*$ with
\[\det \Psi = c \frac{(dz)^2}{z(z-1)(z-m)}.\]
Without loss of generality we can choose $c=1.$ The eigenlines of $\Psi$ are thus well defined on a double covering  of  $\CP^1$ branched at $0, 1, m$ and $\infty,$ which is a complex torus $T^2 = \C/\Gamma.$ Let $\Gamma = \Span\{1+\tau,1- \tau\}$ 
for some $\tau\in \C\setminus\R.$ Without loss of generality the elliptic involution $\sigma$ inducing $\pi : T^2 \rightarrow \CP^1=T^2/\sigma$ is given by $[\omega] \mapsto [-\omega].$ Thus  the preimages of the branch points $0,1,m, \infty \in \CP^1$ are $\omega_0 := [0],$ $\omega_1:=  [\tfrac{1}{2}- \tfrac{\tau}{2}] $, $\omega_2:= [1]$ and $\omega_3 := [\tfrac{1}{2} + \tfrac{\tau}{2}].$ Let $L^{\pm}$ denote the Eigenlines of $\pi^*\Psi.$ Since 
\begin{equation}\label{L+L-}
L^+ \otimes L^- = L^+ \otimes \sigma^*L^+ = L(-\omega_0 -... - \omega_3),
\end{equation}
the eigenlines $L^\pm$ have degree $-2.$ Let $S:= L(-2 \omega_0) (= L(-2\omega_i)).$ Then we have $\sigma^*S = S$ and $S^2 = L(-\omega_0 - ... \omega_3),$ because there exists a meromorphic function on $T^2$ with a pole of order $3$ in $\omega_0$ and simple zeros at the other $\omega_i,$ e.g. the derivative of the Weierstrass $\wp$-function. From \eqref{L+L-}
we get the existence of an $E \in Jac(T^2),$ where Jac$(T^2) \cong \C/\Lambda$ is the moduli space of holomorphic line bundles on $T^2$ of degree $0,$
such that
$$L^+ = S \otimes E, \quad L^- = S \otimes E^*.$$
Let $\nabla^S$ be the connection on $S^*$ such that the connection $\nabla^S \otimes \nabla^S$ on $(S^*)^2$ annihilates the holomorphic section $s_{\omega_0+ ... +\omega_3}$ with simple zeros at $\omega_i.$ Then by construction $\nabla^S$ has  simple poles at $\omega_i$ with residues $-\tfrac{1}{2}.$ A computation in local coordinates (see \cite{HeHe}) shows that with respect to the splitting $V = E\oplus E^*$ the connection $\hat \nabla = \nabla\otimes \nabla^S$ is
gauge equivalent to
\begin{equation}\label{abel_connection}
\hat\nabla=\hat\nabla^{\alpha,\chi}=d+\dvector{\alpha dw-\chi d\bar w&\beta^-\\ \beta^+ &-\alpha dw+\chi d\bar w},\end{equation}
where $w$ is the coordinate on $\C,$ $\alpha,\chi\in\C$ are suitable complex numbers, and
$\beta^\pm=\beta^\pm_\chi$ are meromorphic sections of the holomorphic line bundles given by the holomorphic structures
\[\dbar^\C\pm2\chi d\bar w\] with simple poles at $w_0,..,w_3.$    Using $\vartheta$-functions, we can write down
the second fundamental forms $\beta^\pm$ of $\hat\nabla$ with respect to the decomposition $E\oplus E^*$ explicitly as long as
$L(\dbar-\chi d\bar w)$ is not a spin bundle of $\C/\Gamma$. Note that spin bundles correspond to the exceptional cases where the parabolic structure does not admit a non-zero parabolic Higgs field.  Further,  the residue of $\hat \nabla$ at $w_i$ is given by $$\text{Res}_{\omega_i}\pi^*\hat \nabla = \begin{pmatrix}0 & 2\hat \rho_i - \tfrac{1}{2}\\ 2\hat \rho_i - \tfrac{1}{2}&0\end{pmatrix}.$$

In \cite{HeHeSch} we considered the case where the associated family $\nabla^{\lambda}$ was gauge equivalent to Fuchsian systems with weights $\rho_i := 2\hat \rho_i - \tfrac{1}{2} = \rho$ independent of $i$ (and $\lambda$). In the following we want to consider the case where we have two distinct pairs of weights. More concretely, we want to restrict to the case with $\rho_0 = \rho_2$ and $\rho_1= \rho_3$.
We will see below that this is the appropriate setup to study $(k,l)$-symmetric CMC surfaces.

   \begin{The}[\cite{HeHe}]\label{2:1} Let  $\rho_0=\rho_2,  \rho_1=\rho_3\in]-\tfrac{1}{2},\tfrac{1}{2}[.$ 
 Away from holomorphic line bundles $\dbar-\chi d\bar w$ with 
 $\chi\in \frac{\pi i(1+ \tau)}{2(\tau-\bar\tau)}\Z+\frac{\pi i(1-\tau)}{2(\tau-\bar\tau)}\Z\equiv\tfrac{1}{2}\Lambda,$ \eqref{abel_connection} gives rise to a 2-to-1 correspondence
 between the moduli space $\mathcal A^1(T^2)$ of flat line bundles on  $T^2$ and an open dense subset of the moduli space
 $\mathcal A^2_{\rho_0,..,\rho_3}(\CP^1\setminus\{z_0,..,z_3\})$ of flat $\SL(2,\C)$-connections on the 4-punctured sphere
$\CP^1\setminus\{z_0,..,z_3\}$ with local monodromies (around $z_i$) lying in the conjugacy class determined by $\rho_i.$

 This 2-to-1 correspondence extends to holomorphic line bundles $\dbar-\chi d\bar w$ with 
 $\chi=\gamma\in \tfrac{1}{2} \Lambda$  in the following way: 
 Assume $\alpha(\chi)$ is a holomorphic map on $U\subset \C\setminus(\tfrac{1}{2} \Lambda)$ and $\gamma\in \bar U\cap \tfrac{1}{2}\Lambda.$ Then, the $\SL(2,\C)$-connections $\nabla^{\chi,\alpha(\chi)}$ on $\CP^1\setminus\{z_0,..,z_3\}$ determined by \eqref{abel_connection} extend to $\chi=\gamma$ (after suitable gauge transformations) 
 if and only if $\alpha$ expands around $\chi=\gamma$ as
\begin{equation}\label{a_spin_expansion}
\alpha(\chi)\sim_\gamma \frac{2\pi i}{\tau-\bar\tau}\frac{\mu_{\gamma}}{\chi-\gamma}+\bar\gamma+\,\text{ higher order terms in } \chi,\end{equation}
for \[\mu_\gamma=\left\{\begin{array}{cl} \pm(\rho_0 + \rho_1) & \mbox{if }\gamma\in\Lambda\\ 
\pm (\rho_0- \rho_1) & \mbox{if } \gamma\in\frac{2\pi i}{\tau-\bar\tau}+\Lambda.\\
0 & \mbox{ else } 
\end{array}\right.
\]
Moreover, the induced parabolic structure at $\chi=\gamma$ is
stable, strictly semi-stable or unstable if $\mu_\gamma>0,$ $\mu_\gamma=0$ or $\mu_\gamma<0,$ respectively.
 \end{The}
A proof of theorem can be found in \cite{HeHe}. Note that we express the expansion formula \eqref {a_spin_expansion} 
in terms of the weights on $\C / \Gamma$ instead of the weights $\CP^1\setminus \{z_0, ..., z_3\}$ as in \cite{HeHe}.

\begin{Rem}
Consider a CMC surface whose associated family of flat connections is gauge equivalent to a family 
of Fuchsian systems.
At the Sym points $\lambda_i\in S^1\subset \C_*$ the corresponding  Fuchsian system splits into line bundle connections. Thus its is necessary that the corresponding parabolic structure is strictly semi-stable. Therefore, the existence of parabolic structures with $\mu_\gamma = 0$ is an obstruction for the existence of closed CMC surfaces with the given weights. 
\end{Rem}
We use $(\rho_0, \rho_1,\chi,\alpha)$  to parametrize  flat connections of the form \eqref{abel_connection}
with four simple poles at $\omega_0, ... \omega_3$ on $T^2$, where we consider $\chi$ as a point in $Jac(T^2)$ and the tuple $(\chi,\alpha)$ as a point $\mathcal A^1(T^2),$ the moduli space of flat line bundle connections on $T^2,$ via
$$(\chi, \alpha) \mapsto d+ \alpha dw - \chi d \bar w.$$
By the Mehta-Seshadri Theorem \cite{MSe} and \cite{HeHe} there is a unique $\alpha=\alpha^u_{\rho_0, \rho_1}(\chi)\in\C$
for every $\chi\in \C\setminus(\tfrac{1}{2}\Lambda)$ 
such that the monodromy representation of the connection 
given by
\eqref{abel_connection} for $(\rho_0, \rho_1,\chi,\alpha^u_\rho(\chi))$ is unitarizable.
In fact,  for every $\rho_0, \rho_1 \in]-\tfrac{1}{2},\tfrac{1}{2}[$
this map $\chi\mapsto \alpha^u_{\rho_0, \rho_1}(\chi)$ induces
a real analytic section of the affine holomorphic bundle $\mathcal A^{1}(T^2)\to Jac(T^2)$ which we  denote by 
\begin{equation}\label{alphaurho}
\alpha^{MS}_{\rho_0, \rho_1}\in\Gamma(Jac(T^2)\setminus\{0\},\mathcal A^1 (T^2)).
\end{equation}
satisfying
$\alpha^{MS}_{\rho_0, \rho_1}([\dbar-\chi d\bar w])=[d+\alpha^u_{\rho_0, \rho_1}(\chi)dw-\chi d\bar w].$
Note that $\alpha^u_{\rho_0, \rho_1}$ satisfies the following functional properties:
\begin{equation}\label{lift_alpha}
\begin{split}
\alpha^u_{\rho_0, \rho_1}(\chi+\frac{2\pi i(1+\tau)}{\tau-\bar\tau})&=\alpha^u_\rho(\chi)+\frac{2\pi i(1+\bar\tau)}{\tau-\bar\tau}\bar\tau\\
\alpha^u_{\rho_0, \rho_1}(\chi+\frac{2\pi i(1-\tau)}{\tau-\bar\tau})&=\alpha^u_\rho(\chi)+\frac{2\pi i(1-\bar\tau)}{\tau-\bar\tau}
\end{split}
\end{equation}
for all $\rho_0, \rho_1\in]-\tfrac{1}{2},\tfrac{1}{2}[$ and $\chi\in\C\setminus \tfrac{1}{2s}\Lambda.$
As in the case considered in \cite{HeHeSch} we have the following lemma holds. 
\begin{Lem}\label{a^u_symmetries}
Let $T^2=\C/((1+\tau)\Z+(1-\tau) \Z)$ and $\rho_0, \rho_1\in[0,\tfrac{1}{2}[.$ Then, the section $\alpha^{MS}_{\rho_0, \rho_1}$ in \eqref{alphaurho}
is odd with respect to the involution on $Jac(T^2)$ induced by dualizing.
If $\tau\in i\R$ then $\alpha^{MS}_{\rho_0, \rho_1}$ is real in the following sense:
\[\alpha^u_{\rho_0, \rho_1}(\bar\chi)=\overline{\alpha^u_{\rho_0, \rho_1}(\chi)}\]
for all $\chi\in\C\setminus(\frac{\pi i(1+\tau)}{\tau-\bar\tau}\Z+\frac{\pi i(1-\tau) }{\tau-\bar\tau}\Z).$ 
 \end{Lem}

\begin{The}\label{slitting_tori}
Let $\rho_0, \rho_1\in]0;\tfrac{1}{2}[$ and
let $\lambda\colon\Sigma\to D_{1+\epsilon}\subset\C$ be a double covering of the $(1+\epsilon)$ disc branched over finitely many points
$\lambda_0=0,\lambda_1,..,\lambda_k$ in the closed unit disc. Further, let $\chi\colon\Sigma\to Jac(T^2)$ be an odd map (with respect to the involution $\sigma$ on $\Sigma$ induced by $\lambda$) and $\hat{\mathcal D}$ be an odd meromorphic lift of $\chi$ to $\mathcal A^1(T^2)$  satisfying the following conditions:
\begin{enumerate}
\item $\chi(\lambda^{-1}(0))=0\in Jac(T^2);$
\item $\hat{\mathcal D}$ has a first order pole over $\lambda^{-1}(0);$
\item $\hat{\mathcal D}$ has a first order pole satisfying the condition induced by \eqref{a_spin_expansion}  at every $\xi_0\in\Sigma\setminus\lambda^{-1}(\{0\})$ with $\chi(\xi_0)\in \tfrac{1}{2}\Lambda$ and no further singularities, where $\Lambda=\frac{\pi i (1+ \tau)}{\tau-\bar\tau}\Z+\frac{\pi i(1- \tau)}{\tau-\bar\tau}\Z$ is the lattice generating
 $Jac(T^2)$; 

\item for all $\xi \in \lambda^{-1}(S^1)$ : $\hat{\mathcal D}(\xi)=\alpha^{MS}_\rho(\chi(\xi));$ 
\item  there are two points
$\xi_1,\xi_2\in\lambda^{-1}(S^1)\subset\Sigma$
satisfying
\[\chi(\xi_1)=[\frac{\pi i(1+\tau)}{2\tau-2\bar\tau}]\in Jac(T^2),\,\, \chi(\xi_2) =[\frac{\pi i (1-\tau)}{2\tau-2\bar\tau}] \in Jac(T^2).\]
\end{enumerate}

Then there exists a
conformal CMC immersion
\[f\colon \hat T^2\setminus (l_1\cup l_2 \cup l_3 \cup l_4)\to S^3\]
whose associated family of flat connections (see Theorem \ref{The1}) is determined by
$(\Sigma,\chi,\hat{\mathcal D}),$ where $\hat T^2 = \C / (2 \Z + 2\tau \Z)$ is the double covering of  $ T^2$ and 
 $l_1=\{[t \mid  t\in[0;1]\}$, $l_2=\{[\tfrac{1+ \tau}{2}+ t]\mid  t\in[0;1]\}$, $l_3 = \{[\tau+ t]\mid  t\in[0;1]\}$ and $l_4= \{[\tfrac{1+ 3\tau}{2} + t]\mid  t\in[0;1]\}$ as indicated in Figure \ref{4cut}. 
\end{The}
 \begin{figure}
\centering
\includegraphics[width=0.5\textwidth]{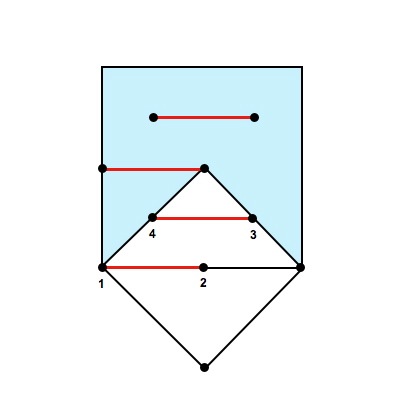}

\caption{
\footnotesize
The white rectangle represents the fundamental piece of the torus $T^2$ which is doubly covered by the blue rectangle representing $\hat T^2$. The $4$ singular points in $T^2$ becomes $8$ singular points on $\hat T^2,$ and the CMC immersion is well-defined
 on $\hat T^2$ without 4 cuts connecting pairs of singular points.}
\label{4cut}
\end{figure}

\begin{proof}
We want to apply Theorem \ref{lifting_theorem} to obtain the CMC surface. 
For $\rho_0, \rho_1 \in ]0, \tfrac{1}{2}[$ the connections given by \eqref{abel_connection} are generically irreducible and $\hat{\mathcal D}$ induces a meromorphic map $\mathcal D$ into the moduli space of flat $\SL(2,\C)$-connections 
$\mathcal{A}^2(T^2 \setminus (l_1 \cup l_2))$.  Since the $\hat{\mathcal D}(\sigma(\xi)) = (\mathcal D(\xi))^*$, where $()^*$ is the gauge equivalence class of the dual connection on the dual line bundle, we obtain $\mathcal D(\sigma \xi) = \hat {\mathcal D}(\xi)$ by \eqref{abel_connection}. By Theorem \ref{2:1} $\mathcal D$ is a well defined map from $D_{1+\epsilon}\setminus\{0\}$ to
$\mathcal{A}^2_{\rho_0,\rho_1,\rho_0,\rho_1}(\CP^1\setminus\{0,1,m,\infty\}).$ 

Condition $(1)$ and $(2)$ ensure that the family of connections has the right asymptotic behavior at $\lambda= 0.$
To see this
 let $\xi^2 = \lambda$ be local coordinates for $\Sigma$ around $\lambda^{-1}(0).$ Then for $\chi(0) = 0,$ there exists by Theorem \ref{2:1} a local map $\alpha_0(\xi)$ with a first order pole at $\xi=0$ such that the the corresponding equivalence class $[\nabla^{\chi,\alpha_0}] $ given by \eqref{abel_connection} is well defined in $\mathcal A^2(T^2\setminus\{\omega_0, ..., \omega_3\})$\footnote{If the map $\mathcal D$ has no pole at $\lambda =0$ then it has also no pole at $\lambda= \infty$. Thus it would be a constant map with $\chi(0) = 0$ and therefore $\mathcal D$ cannot meet condition $(5)$.}.  The difference between $\mathcal D$ and $[\nabla^{\chi,\alpha_0}]$ is given by a family of Higgs fields $\Psi(\xi)$ (in diagonal form w.r.t. $E\oplus E^*$) and first order poles in $\xi$ on the diagonals.  Thus its determinant is given by 
$$\det (\Psi(\xi)) = \tfrac{c^2}{\xi^2} dz + \text{holomorphic terms in }\xi,$$

i.e., it has a first order pole in $\lambda.$
 But $\mathcal D$ and $[\nabla^{\chi,\alpha_0}]$ are well defined (meromorphic) in terms of $\lambda.$ Thus $\Psi(\xi)$ can be meromorphically parametrized in $\lambda.$ Let 
 
  $$\Psi(\lambda) \sim \tfrac{1}{\lambda} \Psi_{-1} + \text{ holomorphic terms in }  \lambda,$$
 
Then det$(\tfrac{1}{\lambda}\Psi_{-1}) \sim  \tfrac{c^2}{\lambda}dz $  yields det$\Psi_{-1} = 0$. Thus $\Psi_{-1}\in\mathfrak {sl} (2, \C)$ has zero determinant and hence it must be nilpotent.
 
By condition $(3)$ there are no further singularities of $\mathcal D$ on $D_{1+\epsilon}$ other than at $\lambda = 0$ and by $(4)$  the map $\mathcal D$ maps into the real subvariety of $\mathcal A^2$ consisting of the equivalence classes of the $SU(2)$ connections. Thus by the Schwarzian reflexion principle there exist a continuation of $\mathcal D$ to $\C_*.$ It remains to show that the Sym point conditions are fulfilled on $\hat T^2\setminus (l_1\cup l_2 \cup l_3 \cup l_4)$, which we proof in the following Lemma.
\end{proof}

 \begin{Lem}\label{Lem:torus-closing}
The condition (5) in Theorem \ref{slitting_tori} ensures that the
monodromy of the corresponding flat
connection $\nabla^{\lambda_j}$ at a Sym point $\lambda_j\in S^1$ is trivial on $\hat T^2\setminus (l_1 \cup l_2\cup l_3 \cup l_4).$
\end{Lem}
 \begin{proof}
 The connection $\nabla^{\lambda_j}$ is unitary and has a strictly semi-stable induced parabolic structure by condition $(5)$ and Theorem \ref{2:1} (see also $\S$2.4 in \cite{HeHe}).
 The corresponding Fuchsian system on $\CP^1\setminus\{0,1,m,\infty\}$ is diagonal and a computation shows that its diagonal entries have the following form
 \[d\pm\hat\rho_0\frac{dz}{z}\pm\hat\rho_1(\frac{dz}{z-1}-\frac{dz}{z-m}),\]
 where $\hat\rho_i=\tfrac{\rho_i}{2}+\tfrac{1}{4}.$ The ''vertical monodromy'' on $\hat T^2\setminus (l_1 \cup l_2\cup l_3 \cup l_4)$
 is then given by the monodromy around the punctures $0$ and $\infty$ on $\CP^1$ and  the ''horizontal monodromy''
 by the monodromy around the punctures $1$ and $m.$ Both of these monodromies are trivial (independently of the actual sign for the line bundle connection).
 \end{proof}

\subsection{Rational weights $\rho_0$ and $\rho_1$}\label{Rational_weights}
We are primary interested in closed surfaces. For rational weights $\rho_0,\rho_1$ we can
guarantee that the analytic continuation of a CMC surface provided by Theorem \ref{slitting_tori}
is automatically closed. Moreover, we can control its genus as well as its branch points and umbilics.

\begin{The}\label{branched_CMC}
Let $\rho_0,\rho_1\in\Q$ with $\tfrac{p_0}{q_0}=\tfrac{2\rho_0+1}{4}$ and $\tfrac{p_1}{q_1}=\tfrac{2\rho_1+1}{4}$ for coprime $p_0,q_0$ and $p_1,q_1,$ respectively.
Set \[k=\begin{cases}
q_0-1\\
\tfrac{q_0}{2}-1\end{cases},\; r_0=\begin{cases}
2p_0& \text{if $q_0$ is   odd}\\
p_0& \text{if $q_0$ is even }\end{cases}\]
and 
\[l=\begin{cases}
q_1-1\\\
\tfrac{q_1}{2}-1 \end{cases}, r_1=\begin{cases}
2p_1& \text{if $q_1$ is   odd}\\
p_1& \text{if $q_1$ is even }\end{cases}.\] 

Then, the analytic continuation  of a CMC surface 
 \[f\colon T^2\setminus (\cup_{i=1}^4 l_i)\to S^3\]
given in Theorem \ref{slitting_tori} is a compact CMC surface of genus
$g=k\cdot l.$ Moreover,
the surface has
$2k+2$
many umbilical branch points $P_i$ of branch order
$l-r_1$
 and umbilical  order $r_1-1,$
and
$2l+2$
many umbilical branch points $Q_j$ of branch order 
$k-r_0$
 and umbilical  order $r_0-1.$
 \end{The}

\begin{proof}
We first prove that the analytic continuation of the CMC surface as given in Theorem \ref{slitting_tori}
 is closed on the $(k,l)$-symmetric covering $M$ of the 4-punctured sphere determined by the elliptic covering $T^2\to\CP^1.$
 Then the statement concerning the genus automatically follows from Riemann-Hurwitz. 
The key point is that the extrinsic closing condition (condition (5) in Theorem \ref{slitting_tori}) guarantees that the connections are reducible (as well as unitary) at the Sym points $\lambda_j$,
see the proof of Lemma \ref{Lem:torus-closing}.
The connection $\nabla^{\lambda_j}$ is thus gauge equivalent to a Fuchsian system on the 4-punctured sphere which reduces
to the direct sum of flat line bundle connections.
 In fact, it can be computed that $\nabla^{\lambda_j}$ is gauge equivalent to
$$  d+ \begin{pmatrix}\pm\frac{n_0}{2l+2}\frac{dz}{z}\pm \frac{n_1}{2k+2}(\frac{dz}{z-1}-\frac{dz}{z-m}) & 0 \\0& \mp\frac{n_0}{2l+2}\frac{dz}{z}\mp \frac{n_1}{2k+2}(\frac{dz}{z-1}-\frac{dz}{z-m}) \end{pmatrix}$$
where $n_0,n_1\in\tfrac{1}{2}\Z\setminus\Z.$
Thus by  tensorizing  the whole family of flat connections $\nabla^{\lambda}$ with the diagonal connection
 \[d^S := d+ \frac{1}{2l+2}\frac{dz}{z}+ \frac{1}{2k+2}(\frac{dz}{z-1}-\frac{dz}{z-m}),\]
we obtain that the connections at the Sym points $\nabla^{\lambda_j} \otimes d^S$ are trivial  by Lemma \ref{Solving-an-ODE}.

The branch orders and umbilic orders of the points $P_i$ and $Q_j$ follows by an analogous
computation as in Theorem 5 of \cite{HeHeSch}.
\end{proof}

\begin{Rem}
We obtain immersed (unbranched) CMC surfaces for spectral data corresponding to $\rho_0 = \tfrac{k-1}{2k+2}$ and $\rho_1 = \tfrac{l-1}{2l+2}$ for all integers $k$ and $ l.$
\end{Rem}
\section{The generalized Whitham flow}\label{whitham flow}

The key issue of the theory is to construct appropriate spectral data satisfying all conditions of Theorem \ref{slitting_tori}. The main problem is to satisfy the unitarity condition for $\nabla^{\lambda}$ along $\lambda \in S^1.$ This is due to the implicitness of the Narasimhan-Seshadri section for higher genus surfaces. In \cite{HeHeSch} we introduced a method to overcome this problem for $\Z_{g+1}$ symmetric surfaces by flowing from know spectral data of a CMC torus towards higher genus data while preserving the unitarity along $S^1$ using the weight $ \rho$ as the flow parameter. In this section we want to generalize this flow to yield  $(k,l)$-symmetric surfaces. We fix the ratio of the weights $\rho_0$ and $\rho_1$ to be $q \in \R_+,$ which we consider as the flow direction. For rational directions we obtain closed but possibly branched surfaces at rational times and immersed surfaces for integer times.  Compared to the situation \cite{HeHeSch} the initial 
CMC immersion is not well-defined on the torus $T^2$ corresponding to the abelianization but only on its double covering $\hat T^2.$ Thus we need to ensure that the spectral data for CMC tori computed in \cite{HeHeSch} are still well defined as a map into the Jacobian of $T^2,$ which double covers the Jacobian of $\hat T^2.$ Because Lemma \ref{a^u_symmetries} holds for rhombic tori, we can use the same function spaces as in \cite{HeHeSch} and the short time existence of the flow for every given flow direction follows analogously. 

\begin{Rem}
The CMC surfaces with $\Z_{g+1}$ symmetry constructed in \cite{HeHeSch} can be considered as a special case of the Theorems \ref{slitting_tori} and \ref{branched_CMC} for $\rho_0= \rho$ and $\rho_1 = 0.$ In particular, they all have an additional commuting $\Z_2$-symmetry with $2g+2$ fix points.
\end{Rem}

\subsection{Flowing homogeneous CMC tori}\label{Tori_spec_gen_0}
Homogenous tori are the simplest CMC tori in $S^3.$ They are given by the product of two circles with different radii and can be parametrized by 
$$f(x,y) = (\tfrac{1}{r} e^{ir x}, \tfrac{1}{s} e^{isy}) \subset S^3 \subset \C^2, \quad r,s \in \R, \quad r^2+ s^2 = 1.$$
Thus (simply wrapped) homogenous tori always have rectangular conformal types and 
we identify $\hat T^2=\C/\Gamma$ where $\Gamma=2\Z+2\tau \Z$ for some
$\tau\in i \R^{\geq1}$. The Jacobian of $\hat T^2$ is  given by 
$Jac(\hat T^2)=\C/(\tfrac{\pi i}{2 \tau}\Z+\tfrac{\pi i}{2}\Z)$

Homogenous tori are of spectral genus $0$ thus their spectral curve $\Sigma$ is given by the algebraic 
equation
\[\xi^2=\lambda.\]
Since $\CP^1$ is simply connected, every meromorphic map
\[\chi\colon \CP^1  \to Jac(\hat T^2) = \C / \hat \Lambda\]
lifts to  $\C$ as a meromorphic function
$\hat{\chi}$. 

The torus $T^2$ used for the abelianization is given by $\C / ((1+ \tau) \Z + (1-\tau)\Z) $ and is doubly covered by $\hat T^2.$ Thus $\chi$ is also a well defined map into the bigger Jacobian $Jac(T^2).$
In \cite{HeHeSch} we have computed $\chi$ (after a shift) to be 
\begin{equation}\label{hom_tori_spec}
\hat{\chi}(\xi)=\tfrac{\pi i R}{4\tau}\xi d\bar w
\end{equation}
for $R= \sqrt{1+\tau\bar\tau}$.

The extrinsic closing condition is that
there exist of four points
\[\pm\xi_1,\pm\xi_2\in S^1\subset\C^*\]
with
\[\hat{\chi}(\pm\xi_{1,2})\in \tfrac{\pi i (1+\tau)}{4\tau}d\bar w+\hat \Lambda.\]

This is exactly the closing condition (5) in Theorem \ref{slitting_tori}, since $\hat \Lambda= \tfrac{\pi i}{2 \tau}\Z+\tfrac{\pi i}{2}\Z.$
 The proof of the following theorem works analogously to the proof of Theorem $6$ in \cite{HeHeSch}.
\begin{The}\label{flow_homog}
Let $q \in R_+$ and $\tau\in i\R^{\geq1}.$ Further, consider the homogenous CMC torus $f: \hat T^2 = \C / (2\Z + 2\tau \Z) \rightarrow S^3$  with $\chi(\xi)=R_0\tfrac{\pi i}{4\tau}\xi,$ where
 $R_0=\sqrt{1+\tau\bar\tau} \in [\sqrt{2}, 2[$ and $T^2 = \C/ (1+\tau)\Z+ (1-\tau)\Z).$
Moreover, let $\rho_0 = t$ and $\rho_1 = q  t$. Then for $t \sim 0$ there exist  a unique
 $\chi^t\colon\{\xi\mid |\xi|<1+\epsilon\}\to Jac(T^2)$ 
 and a lift $\mathcal D^t\colon\{\xi\mid |\xi|<1+\epsilon\}\to\mathcal A^1(T^2)$
 satisfying the conditions of Theorem \ref{slitting_tori}.
 
Thus there exists a flow of closed CMC surfaces $f^t$ with boundaries. For rational $t,q\in\Q$ the analytic continuation of the surfaces are closed and give compact (possibly branched) CMC surfaces. 
 \end{The}
 \begin{Rem}\label{3d}
 Instead of fixing the ratio $q=\tfrac{\rho_1}{\rho_0}$ and using $t=\rho_0$ as a flow parameter, we could
 have considered $\rho_0$ and $\rho_1$ as two independent flow parameters for commuting flows. Hence, locally
 around the (stable) homogeneous CMC tori, there is a 2-dimensional space of admissible spectral data for CMC surfaces
 of the "same conformal type" determined by the branch images $0,1,m,\infty\in\CP^1.$ If we would also allow the conformal type to change, we would get a
 3-dimensional space with local parameters $(\rho_0,\rho_1,m)$ with $m\in S^1\subset\C$
 (corresponding to rhombic tori).
 \end{Rem}
   \begin{figure}
\centering
\includegraphics[width=0.23\textwidth]{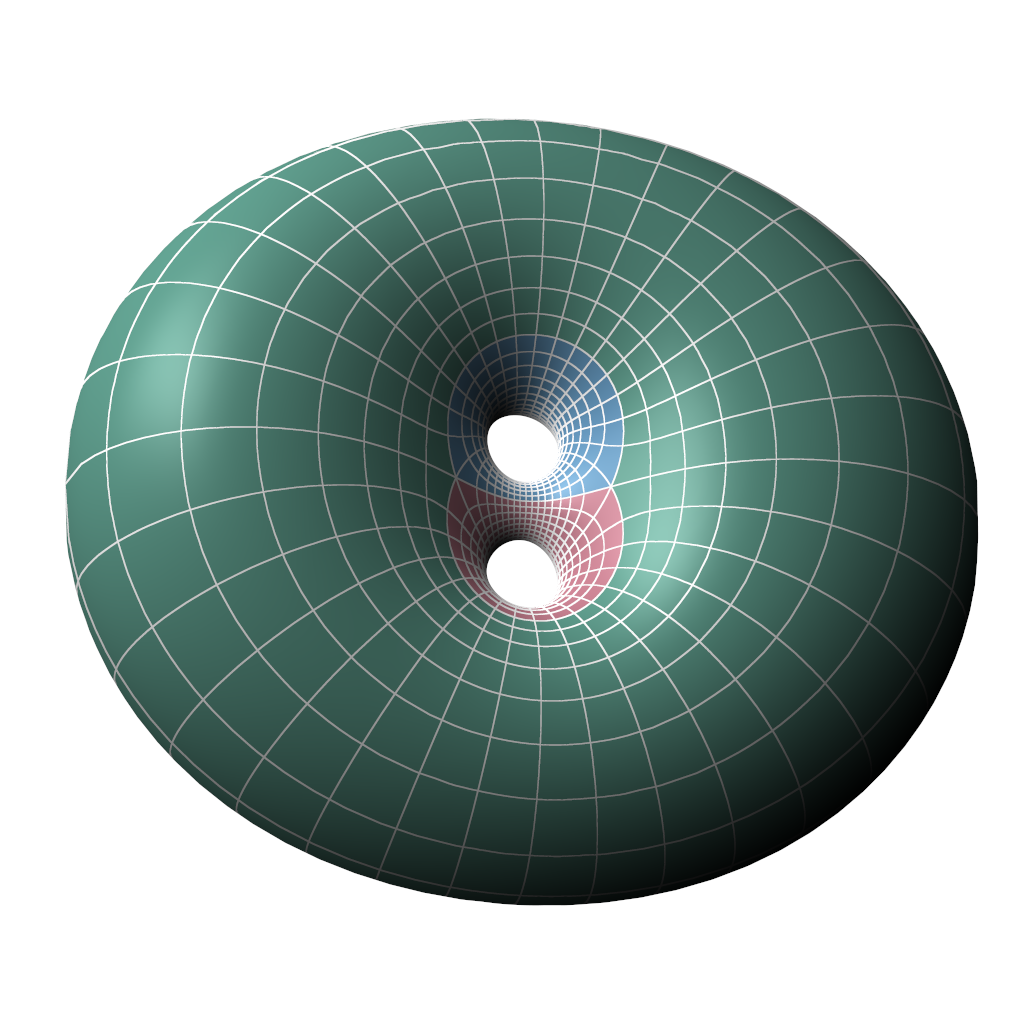}
\includegraphics[width=0.22\textwidth]{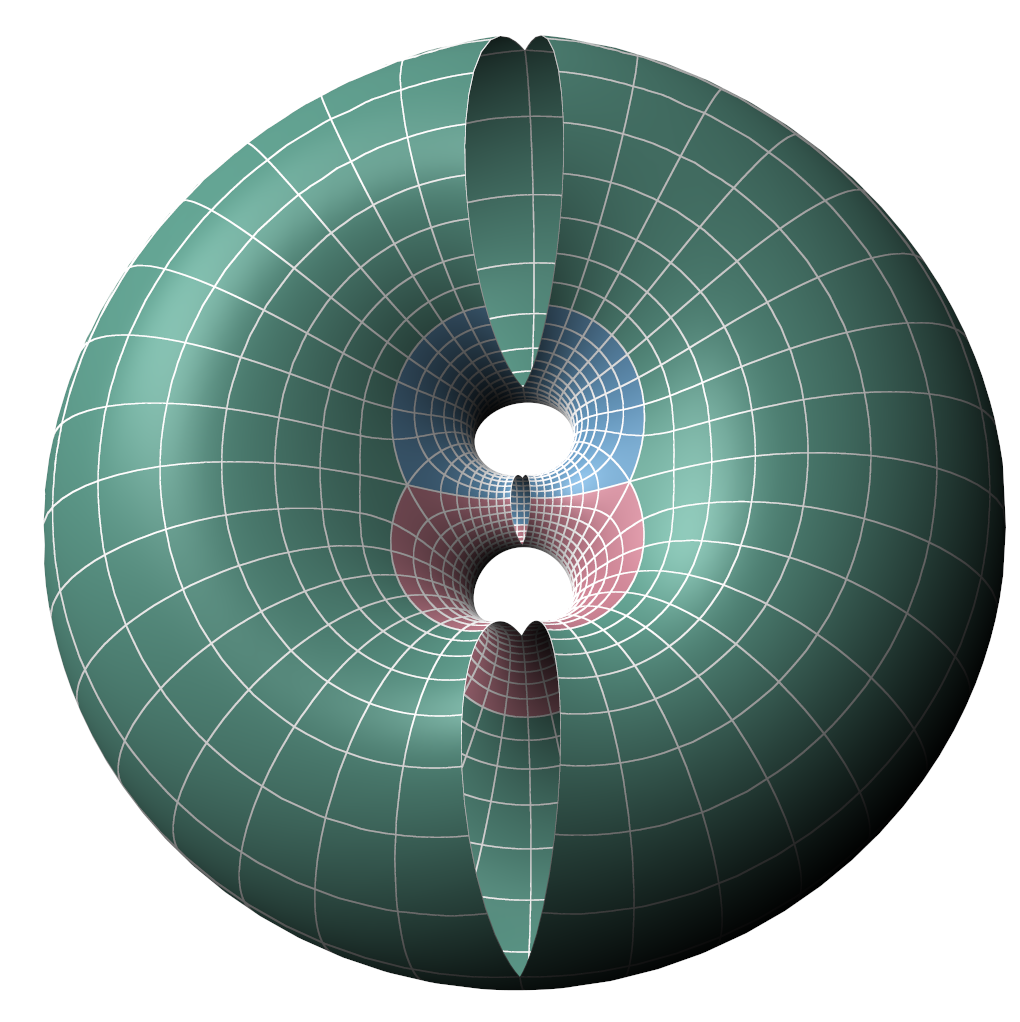}
\includegraphics[width=0.22\textwidth]{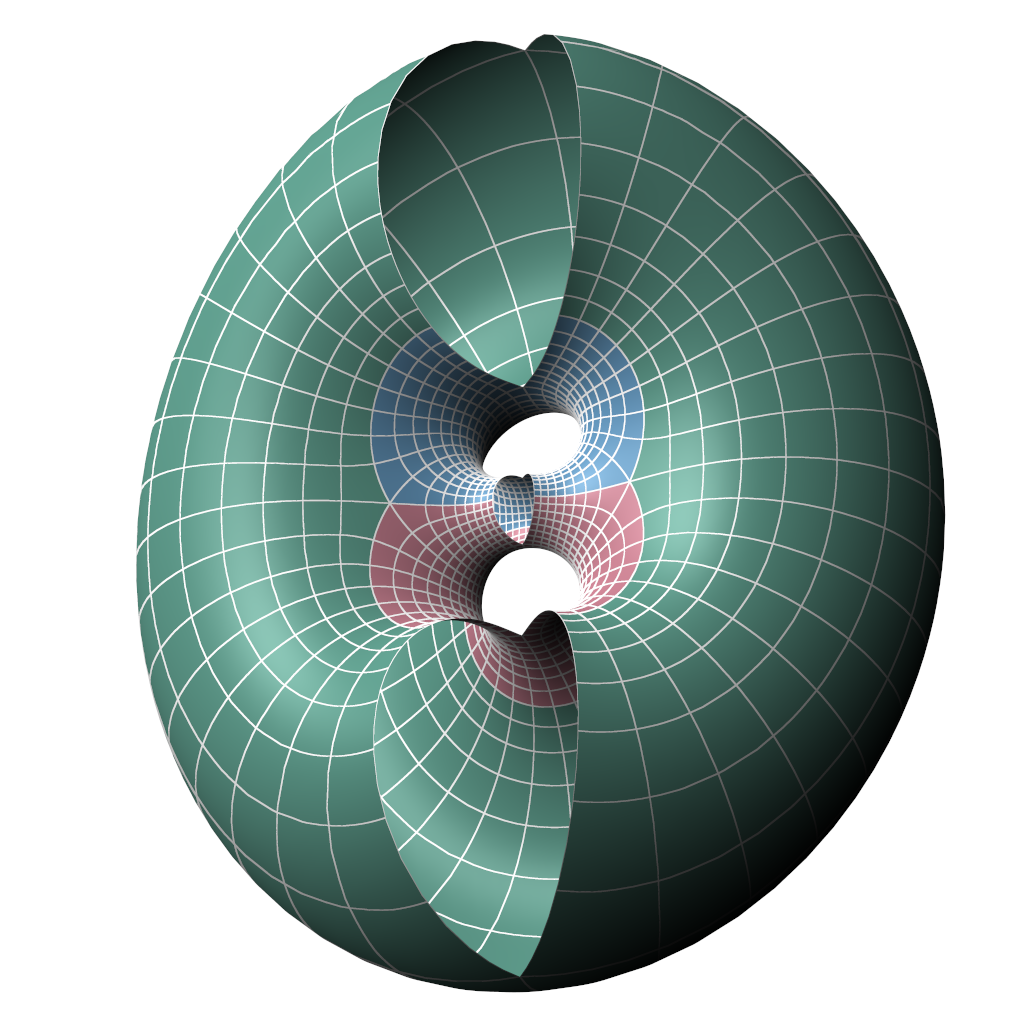}

\includegraphics[width=0.24\textwidth]{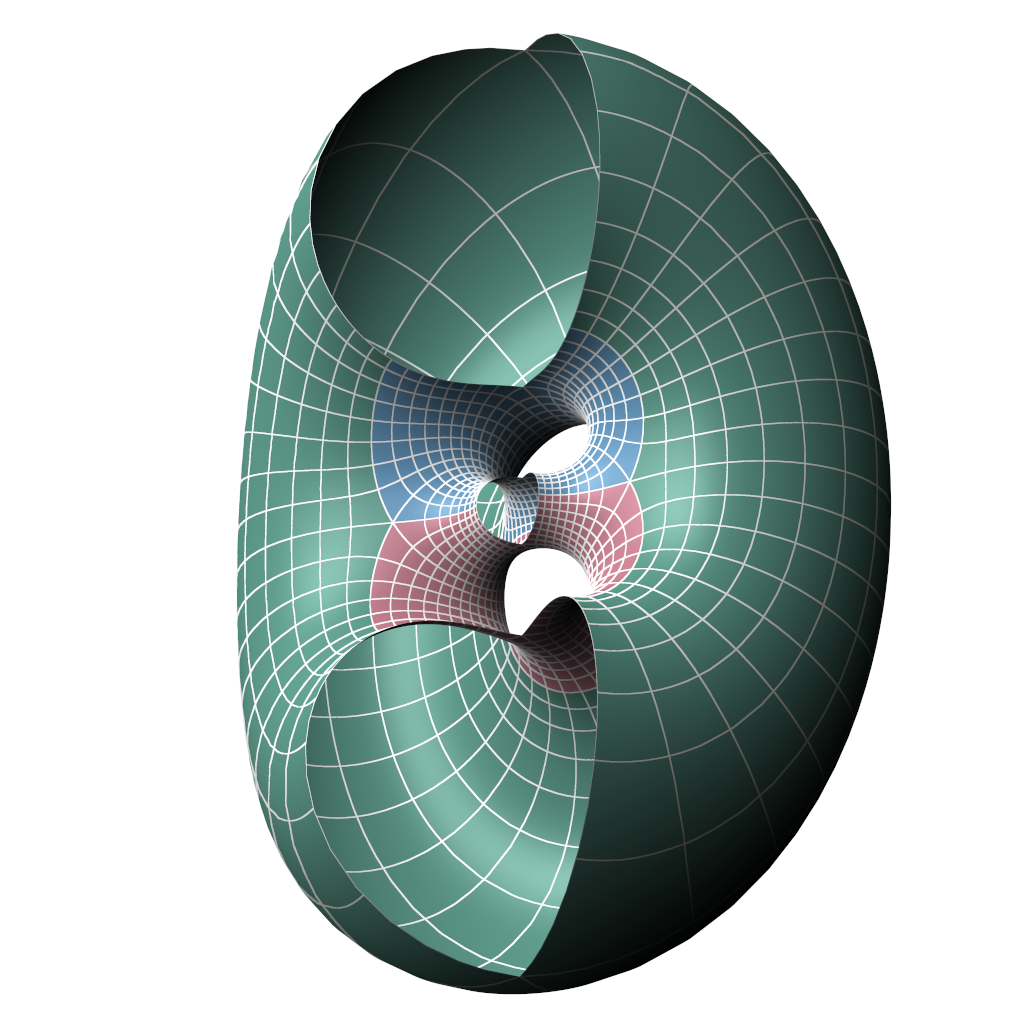}
\includegraphics[width=0.24\textwidth]{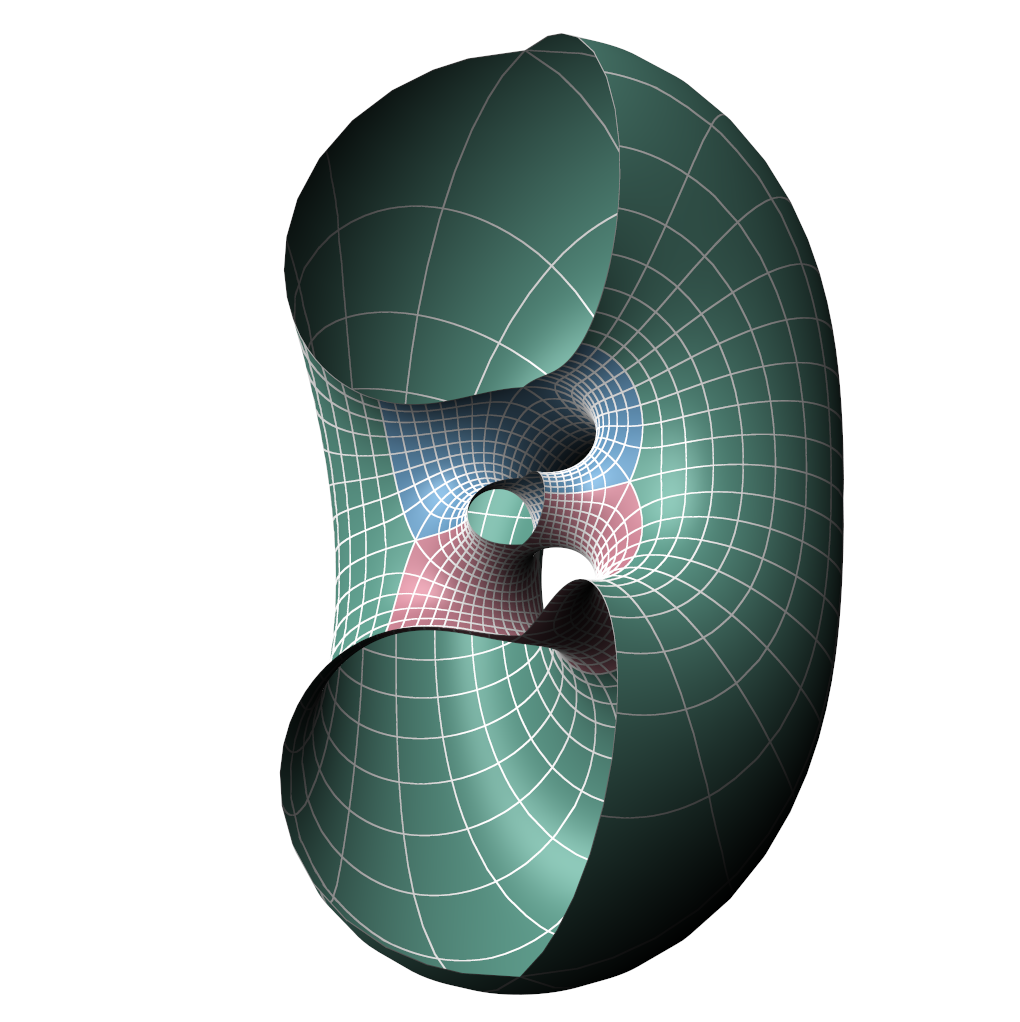}
\includegraphics[width=0.24\textwidth]{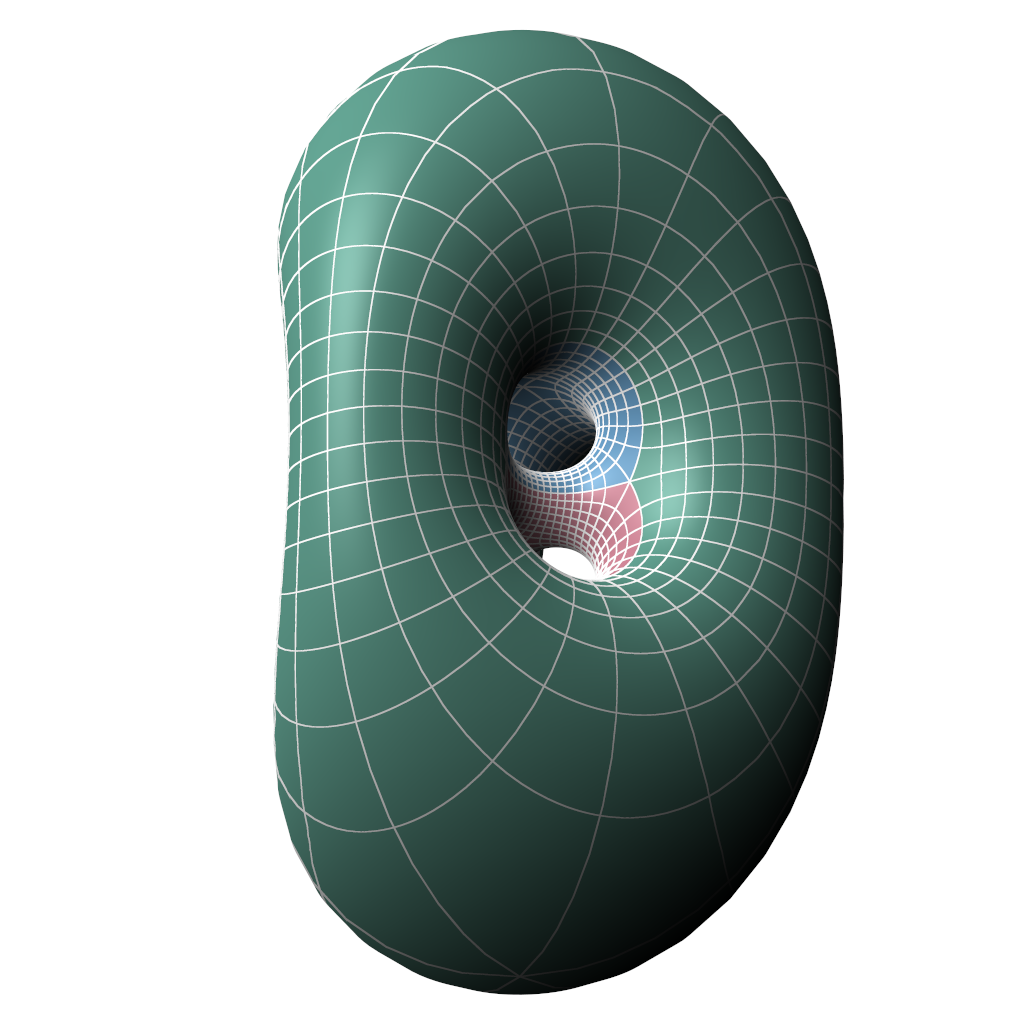}

\caption{
\footnotesize
The deformation of minimal surfaces from the Lawson surface $\xi_{2,1}$ of genus $2$ to $\xi_{2,2}$ of genus $4$. An order $2$ symmetry of the initial surface with six fixed points is ÒopenedÓ along three cuts until it reaches an order $3$ symmetry. The final surface is shown before and after the missing piece is filled in.
}
\label{fig:lawson}
\end{figure}

\begin{Exa}[ The Lawson surfaces $\xi_{k,l}$ ]
Choose $(k, l)\in \N^2$ and let $q = \tfrac{(l-1)(2k+2)}{(2l+2)(k-1)}.$
We consider  the corresponding generalized Whitham flow starting at the Clifford torus which has square conformal structure. This implies that there is an
additional symmetry on the Riemann surface $\hat T^2= (\C/ 2\Z + 2i\Z)$ and its Jacobian which is given by the multiplication by $i$. The Riemann surface $T^2$ is also a square torus and inherits the extra symmetry. Thus by a similar argument as in Example 3.1 of \cite{HeHeSch} we obtain that all surfaces $f^{t}$ within the flow starting at the Clifford torus are minimal. Let $S \subset \hat T^2=\C/(2\Z+2i\Z)$ be the square defined by the vertices $[0],[\tfrac{1+i}{2}], [1], [\tfrac{(1-i)}{2}] \in \hat T^2,$ see Figure \ref{4cut}. Since the coordinate lines of $\hat T^2$ are curvature lines of the corresponding surfaces (by reality of the Hopf differential), the diagonals are asymptotic lines. 
Because of the reality of the spectral data, the diagonals are geodesics on the surface.
Thus the image of the boundary of $S$ under the minimal immersion $f^{t}$ is a geodesic 4-gon. For $t=\tfrac{k-1}{2k+2}$ it is the geodesic polygon used in \cite{L} to construct the Lawson surfaces $\xi_{k,l}$, see also Figure \ref{fig:lawson}.

\end{Exa}

\subsection{Flowing $2$-lobed Delaunay Tori}\label{Tori_spec_gen_1}
Now we turn to the other class of initial surfaces considered in \cite{HeHeSch}, namely the $2$-lobed CMC tori of revolution. These are obtained  by the rotation of an elastic curve (with intrinsic period $2$)  in the upper half plane, viewed as the hyperbolic plane  $H^2,$ around the $x$-axis, where we consider $S^3$ as the one point compactification of $\R^3,$ see \cite{LHe, KSS}.  The conformal type of a CMC torus of revolution is again rectangular and determined by a lattice $\hat \Gamma=2\Z+2\tau \Z$ for some
$\tau\in i \R^{>0}$. The spectral curve of a CMC torus of revolution has genus $1$ and has rectangular conformal type. Thus it can be identified with
\[\Sigma=\C/(\Z+\tau_{spec}\Z),\] where $\tau_{spec}\in i\R.$
Let $\xi$ be the coordinate on the universal covering $\C$ of $\Sigma.$
Then it is shown in \cite{HeHeSch} that
 $\chi([0])=0 \in Jac(\hat T^2)$
is the trivial holomorphic line bundle and that  $d\chi$  satisfies 
\begin{equation}\label{weierstrass_zeta_L}
d\chi=(a\wp(\xi-\tfrac{\tau_{spec}}{2})+b) d\xi\otimes\frac{\pi i}{2\tau} d\bar w,
\end{equation}
where $\wp$ is the Weierstrass $\wp$-function on $\Sigma$ and $a$ and $b$ are real parameters uniquely determined by 
\begin{equation}\label{2lobe_spec_periods}
\begin{split}
\int_{\tau_{spec}}(a\wp(\xi-\tfrac{\tau_{spec}}{2})+b) d\xi&=0,\\
\int_{1}(a\wp(\xi-\tfrac{\tau_{spec}}{2})+b) d\xi&=2.
\end{split}
\end{equation}

Since $\tfrac{\pi i}{ \tau}$ is a lattice point of $Jac(T^2) = \C/ (\tfrac{\pi i(1+\tau)}{2\tau}\Z +\tfrac{\pi i (1-\tau)}{2\tau}\Z)),$  the map  $\chi$ is also well defined as a map from $\Sigma$ to $Jac(T^2).$ 

The Sym point condition is that there exist $\xi_{1/2} \in \Sigma$ such that 
\[\chi(\xi_{1/2})\in \frac{\pi i (1\pm\tau)}{4\tau}d\bar w+\hat\Lambda,\]
which is again exactly the Sym point condition we pose in Theorem \ref{slitting_tori}.
 Again the methods used to prove Theorem 7 in \cite{HeHeSch}  apply in our situation here and yield the following theorem.
\begin{The}\label{flow_Delaunay}
Let $q \in R_+$ and $\tau\in i\R$  give the conformal type of a 2-lobed Delaunay CMC torus with $\chi\colon\Sigma\to Jac(T^2)$ as described above.

Further, let $\rho_0 = t$ and $\rho_1 = q t$. Then for $t \sim 0$ there exist a unique 
rectangular
 elliptic curve $\lambda_+\colon\Sigma^+\to\CP^1$ together with an odd holomorphic map
 \[\chi^t_+\colon\lambda_+^{-1}(\{\lambda\mid |\lambda|<1+\epsilon\})\subset\Sigma^+\to Jac(T^2)\]
 and a lift $\mathcal D^t_-\colon\lambda_+^{-1}(\{\lambda\mid |\lambda|<1+\epsilon\})\to\mathcal A^1(T^2)$
 satisfying the conditions of Theorem \ref{slitting_tori} such that 
  the induced parabolic structure is stable (see Theorem \ref{2:1}) for all $\lambda$ inside the unit disc.

Moreover, for $t \sim 0$, there exist a unique rectangular
 elliptic curve $\lambda_{-}\colon\Sigma^{-}\to\CP^1$ together with an odd holomorphic map
 \[\chi^t_-\colon\lambda_{-}^{-1}(\{\lambda\mid |\lambda|<1+\epsilon\})\subset\Sigma^-\to Jac(T^2)\]
 and a lift $\mathcal D^t_-\colon\lambda_{-}^{-1}(\{\lambda\mid |\lambda|<1+\epsilon\})\to\mathcal A^1(T^2)$
 satisfying the conditions of Theorem \ref{slitting_tori} such that 
 there is exactly one $\lambda$ inside the unit disc where the induced holomorphic structure is unstable.

Thus for every $q\in\R_+$  there exists two distinct flows starting at the 2-lobed Delaunay tori through closed CMC surfaces with boundaries in the sense of Theorem \ref{slitting_tori}. For rational $t$ and $q$ the analytic continuation of the corresponding surface is closed and yield a compact (possibly branched) CMC surface. 
\end{The}


\end{document}